\documentclass[12pt]{amsart}
\usepackage{preamble}
\usepackage{romannum}
\usepackage[headheight=15pt, headsep=15pt, footskip=27pt, bottom=2.4cm, left=2.4cm, right=2.4cm]{geometry}

\usepackage[pagewise]{lineno}

\begin{document}

\pagenumbering{arabic}

\title[Varieties with nef anticanonical divisors]{Varieties with nef anticanonical divisors and Albanese morphisms of relative dimension one in positive characteristic}


\begin{abstract}
Let $X$ be a smooth projective variety with a nef anticanonical divisor over an algebraically closed field of characteristic $p>0$. In this paper, we establish a precise structure of $X$ under the condition that $a_X: X \to \Alb(X)$ is of relative dimension one.
\end{abstract}


\author{Tongji Gao}
\address[Tongji Gao]{Department of Mathematics, Southern University of Science and Technology, 1088 Xueyuan Rd, Shenzhen 518055, China.} \email{12231277@mail.sustech.edu.cn}

\author{Zhan Li}
\address[Zhan Li]{Department of Mathematics, Southern University of Science and Technology, 1088 Xueyuan Rd, Shenzhen 518055, China.} \email{lizhan@sustech.edu.cn}

\author{Lei Zhang}
\address[Lei Zhang]{School of Mathematical Science, University of Science and Technology of China, Hefei 230026, China.} \email{zhlei18@ustc.edu.cn}

\maketitle

\tableofcontents

\section{Introduction}\label{sec: intro}

Positivity of anticanonical divisors imposes strong restrictions on the geometry of algebraic varieties. Over the field of complex numbers, \cite{Cal57, Kaw85, Amb05} show that the Albanese morphism of a Calabi-Yau variety is a locally trivial fibration. This phenomenon is also conjectured to hold for the Albanese morphism of a K\"ahler manifold with nef anticanonical divisor \cite{DPS96}, and has been affirmatively established for projective manifolds by \cite{Cao19} and for K\"ahler manifolds by \cite{MWWZ25}. For similar results on log pairs with singularities, see \cite{Wan22, MW23}. Additionally, the MRC fibrations of such varieties exhibit analogous properties \cite{CH19, MW23, MWWZ25}. These structures are closely related to the (numerically) non-vanishing of anticanonical divisors \cite{LMPTX22, Mul23}.

In positive characteristic, similar results have been established assuming that the fibers have good singularities \cite{PZ19, Eji23, EP23}. However, wild fibrations in positive characteristic prevent the application of techniques for such generalizations. On the other hand, the wild fibration is a unique phenomenon of algebraic geometry in positive characteristic, setting it apart from algebraic geometry over complex numbers. The authors attempt to understand this phenomenon in \cite{CWZ23, CWZ24}. They first focus on varieties with nef anticanonical divisors and the Albanese morphism of relative dimension $1$. They prove that for such a variety $X$, the Albanese morphism $a_X: X \to A$ is a fibration, and for their purposes they establish a very precise structure of a threefold $X$ with $K_X \equiv 0$ \cite[Theorem 1.1]{CWZ24}: roughly speaking, there exists an isogeny $A' \to A$ such that the normalization of the reduced fiber product satisfies
\[
(X \times_A A')_{\mathrm{red}}^\nu \cong A' \times F,
\]
where $F$ is either an elliptic curve or $\mathbb{P}^1$.

Following \cite{ CWZ24}, the present note completes the last piece of this series of works by studying the explicit structures of the Albanese morphism of relative dimension $1$.
To be more precise, we work in the following setting.

\medskip

\begin{quotation}
\noindent {\bf (Setting $\dagger$)} Let $X$ be a smooth projective variety over an algebraically closed field $k$ of characteristic $p>0$. Assume that 
\begin{enumerate}
\item the resolution of singularities holds in dimension $<\dim X$,
    \item $-K_X$ is nef but not numerically trivial, and 
        \item the Albanese morphism $a_X: X \to A$ is of relative dimension one. 
\end{enumerate}
\end{quotation}

\medskip

Condition (3) implies that the arithmetic genus of the generic fiber $p_a(X_\eta)=0$ by the canonical bundle formula in \cite[Theorem 1.2, 1.3]{CWZ23}.

\begin{theorem}[{Theorem \ref{thm: structure-sep1} \& Theorem \ref{thm: structure-sep2}}]\label{thm: 1}
    In {\bf (Setting $\dagger$)}, if $a_X$ is separable, then there exists an isogeny $B \to A$ between abelian varieties of degree at most $2$ , such that 
    \[
    X \times_A B \simeq \Pp_B(\Gg)
    \] where $\Gg$ is a numerically flat vector bundle of rank $2$.

Besides, we have $h^0(X, \Oo_X(-2K_X)) > 0$. 
\end{theorem}

\begin{theorem}[{Theorem \ref{thm: structure-inseparable} \& Theorem \ref{thm: classfy F}}]\label{thm: 2}
    In {\bf (Setting $\dagger$)}, if $a_X$ is inseparable, then
    
(i) there exists an isogeny $A' \to A$ of abelian varieties, which is purely inseparable of degree 4, such that $X\times_A A'$ is not reduced, and the normalization $X'$ of $(X\times_A A')_{\rm red}$ is isomorphic to $A'\times \Pp_k^1$;

(ii) $X\cong (A'\times \Pp_k^1)/\Ff$ where $\Ff$ is a smooth rank one foliation of $A'\times \Pp_k^1$ such that
$\Ff \sim h^*\mathcal{O}_{\Pp^1}(-1)$ ($\Ff$ has a full description);
    
(iii) $h^0(X, \Oo_X(-2K_X))>0$.
\end{theorem}

We now outline the contents of this paper. Section~\ref{sec: pre} provides essential background material and sets up the notation used throughout. Section~\ref{sec: related results} surveys recent results concerning varieties with nef anticanonical divisors in characteristic $p$. In Section~\ref{sec: geo reduced fibers}, we examine the case where the Albanese morphism is separable, distinguishing three subcases. For the first two, we establish a projective bundle structure on $X$ and prove the non-vanishing of $-K_X$ (Theorem~\ref{thm: structure-sep1}). In the third case, we show that, after an isogeny base change of degree $2$, the situation reduces to the preceding cases (Theorem~\ref{thm: structure-sep2}). Section~\ref{sec: geo non-reduced fibers} is devoted to the study of inseparable Albanese morphisms. We prove that, after an isogeny base change, passage to the reduced scheme structure, and normalization yield a product structure $A'\times \Pp_k^1$, namely, $X$ arises as the quotient of $A' \times \mathbb{P}_k^1$ by a rank one foliation $\mathcal{F}$ (Theorem~\ref{thm: structure-inseparable}). Furthermore, when $X$ is threefold, we classify all possible such foliations on $A' \times \mathbb{P}_k^1$ (Theorem~\ref{thm: classfy F}).

\medskip
\subsection*{Acknowledgments} Tongji Gao and Zhan Li are partially supported by the NSFC (No.12471041) and the Guangdong Basic and Applied Basic Research Foundation (No.2024A1515012341). Lei Zhang is partially supported by NSFC (No.\ 12122116 and No. 12471495) and CAS Project for Young Scientists in Basic Research, Grant No.\ YSBR-032. We thank Ni An for discussing foliations and Vladimir Lazi\'c for pointing out the reference \cite{Mul23}.

\section{Preliminary}\label{sec: pre}

In this section, we gather relevant concepts and results that will be used later.

\subsection{Notation and convention}
Let $k$ be an algebraically closed field of $\chara k=p>0$. A variety $X$ over $k$ is a separated integral scheme of finite type over $k$. A curve is a $k$-variety of dimension $1$. For a $\Qq$-Cartier divisor $D$ on $X$, we write $\kappa(X, D)$ (resp. $\nu(X, D)$) for the Kodaira dimension (resp. numerical dimension) of $D$. We also write $D \succcurlyeq 0$ if $\kappa(X, D) \geq 0$. A projective morphism $f: X \to S$ between normal varieties is called a fibration if $f_*\Oo_X=\Oo_S$. For a base change $T \to S$, we set $X_T \coloneqq X \times_S T$. 

A variety $X$ is of maximal Albanese dimension if $\dim a_X(X) = \dim X$, where $a_X: X \to A$ is the Albanese morphism.

\subsection{Foliations and purely inseparable morphisms}\label{sec: foliation}
Let $Y$ be a normal variety over $k$.
By a {\it foliation} on $Y$ we mean a saturated subsheaf $\mathcal F\subseteq \mathcal T_Y$ of the tangent bundle that is $p$-closed and involutive. Let
$$\mathrm{Ann}(\mathcal F):=\{g\in \mathcal{O}_Y~|~\partial(g) =0, \forall \partial \in \mathcal F\} \subseteq \mathcal{O}_Y$$
which is a subsheaf of ring containing $\mathcal{O}_Y^p$.
There is a one-to-one correspondence  (cf. \cite{Eke87}): \[
    \newcommand\mysatop[2]{\genfrac{}{}{0pt}{}{#1}{#2}}
    \left\{\mysatop{\text{foliations}}{\mathcal F\subseteq \mathcal T_Y}\right\}
    \leftrightarrow
    \left\{\vcenter{\hbox{finite purely inseparable morphisms $\pi\colon Y\to X$}%
                    \hbox{over $k$ of height one with $X$ normal}}\right\},
\]
which is given by
$$\mathcal F ~\mapsto ~\pi\colon Y \to \mathrm{Spec}(\mathrm{Ann}(\mathcal F)) \text{ \ and \ } \pi\colon Y\to X~\mapsto \mathcal F_{Y/X}:=\Omega_{X\to Y}^{\perp}$$
where $\Omega_{X \to Y}:=\mathrm{im }(\pi^*\Omega_X^1 \to \Omega_Y^1)$ and $\Omega_{X\to Y}^{\perp}$ is the sheaf of tangent vectors in $\mathcal{T}_Y$ annihilated by $\Omega_{X \to Y}$.
As a side note, $(\Omega_{Y/X}^1)^\vee \cong \mathcal F_{Y/X}$, and under the above correspondence $\mathrm{rank} (\mathcal F_{Y/X}) =\log_p\deg f$.
Let $y\in Y$ be a smooth point and $x:=\pi(y)$.
A foliation $\mathcal{F}\subseteq\mathcal T_Y$ is said to be \emph{smooth} at $y$ if around $y$ the subsheaf $\mathcal F$ is a subbundle, namely both $\mathcal{F}$ and $\mathcal{T}_{Y}/\mathcal{F}$ are locally free.
It is known that (\cite[p.142]{Eke87}) \[
  \hbox{$\mathcal F$ is smooth at $y$ $\Leftrightarrow$ $Y/\mathcal F$ is smooth at $x$ $\Rightarrow$ $\Omega^1_{Y/X}$ is locally free at $y$} .
\]

First, recall the following well-known result (cf. \cite{Eke87}).
\begin{proposition}\label{prop:can-foliation}
  Let $\pi\colon Y\to X$ be a finite purely inseparable morphism of height one between normal varieties.
  Then
  \begin{equation}\label{eq:pullback-cano}
    \pi^*K_X \sim K_Y - (p-1)\det \mathcal F_{Y/X}\sim  K_Y + (p-1)\det \Omega_{Y/X}^1.
  \end{equation}
\end{proposition}

To treat inseparable fibrations, we shall use the following result (\cite[Proposition 3.3]{CWZ23}),  which is a slight generalization of \cite[Theorem~1.1]{JW21}.
\begin{proposition}\label{prop: CWZ23prop3.3}
  Let $f\colon X\to S$ be a fibration, $\tau\colon T\to S$ a finite purely inseparable morphism of height one.
Let $Y$ be the normalization of $(X_T)_{\rm red}$.
Consider the following commutative diagram
$$\xymatrix{&Y\ar[r]^<<<<<\nu \ar@/^8mm/[rrr]|{\,\pi\,} \ar[rrd]_g\ar[r] &(X_T)_{\rm red}\ar[r] &X_T\ar[r]\ar[d]^{f_T} &X\ar[d]^{f}\\
&&  &T\ar[r]^{\tau} &S\rlap{.} \\
}$$
  Let $\Gamma \subseteq H^0(T,\Omega_{T/S}^1)$ be a finite-dimensional $k$-vector subspace.
  Assume that there is an open subset $U\subseteq T$ such that $\Omega_{U/S}^1$ is locally free and globally generated by $\Gamma$.
  Set $r= \mathrm{rank} \Omega_{Y/X}^1$ and $\Gamma_Y= \mathrm{Im}\bigl(\bigwedge^r \Gamma \to H^0(Y,\det \Omega_{Y/X}^1)\bigr)$.
  Let $\mathfrak M +F \subseteq \lvert \det \Omega_{Y/X}^1 \rvert$ be the sub-linear system determined by $\Gamma_Y$ with the fixed part $F$ and the movable part $\mathfrak M$.
  Then
  \begin{itemize}
    \item[(1)] $\nu(F)|_{(X_U)_{\mathrm{red}}}$ is supported on the codimension one part of the union of the non-normal locus of $(X_U)_{\rm red}$ and the exceptional locus over $U$;
    \item[(2)] the $g$-horizontal part $M_{h}$ of $M \in \mathfrak M$ is zero if and only if $X_{K(T)}$ is reduced.
  \end{itemize}
\end{proposition}

\subsection{Adjunction formula}

\begin{theorem}\label{thm: adjunction}
  Let $X$ be a normal variety and $S$ a prime Weil divisor of $X$.
  Let $S^\nu \to S$ be the normalization. Assume that $K_X + S$ is $\mathbb{Q}$-Cartier.
  Then
  \begin{enumerate}
    \item[(i)] There exists an effective $\Qq$-divisor $\Delta_{S^\nu}$ on $S^\nu$ such that
      \[
      (K_X+ S)|_{S^{\nu}} \sim_\Qq K_{S^{\nu}} + \Delta_{S^\nu}.
      \]
    \item[(ii)] Let $V \subset S^\nu$ be a prime divisor, then $\mathrm{coeff}_V \Delta_{S^\nu} = 0$ if and only if $X,S$ are both regular at the generic point of $V$.
    \item[(iii)] If the pair $(S^\nu, \Delta_{S^\nu})$ is strongly $F$-regular, then $S$ is normal.
  \end{enumerate}
\end{theorem}
\begin{proof}
Please refer to \cite[Proposition 4.5]{Kol13} for statements (i) and (ii) and to  \cite[Theorem 1.4]{Das15} for (iii).
\end{proof}

\subsection{Numerical flat vector bundles}
\begin{definition}\label{def: num flat}
    Let $\Ee$ be a vector bundle on a projective variety. Then $\Ee$ is said to be numerically flat if $\Ee$ and $\Ee^\vee$ are both nef.
\end{definition}

Numerically flat vector bundles on abelian varieties have a particularly simple structure, as described by the following result.

\begin{theorem}[{\cite[Theorem 5.4]{Eji23}}]\label{thm: num flat on abelian}
    Let $\Ee$ be a vector bundle on an abelian variety $A$. Then the following are equivalent:
    \begin{enumerate}
        \item $\Ee$ is numerically flat;
        \item $\Ee$ is obtained as iterated extensions of numerically trivial line bundles;
        \item $\Ee\simeq\bigoplus _iU_i\otimes\Ll_i$, where each $U_i$ is a unipotent vector bundle and each $\Ll_i$ is an numerically trivial line bundle.
    \end{enumerate}
\end{theorem}



  
\subsection{Curves with arithmetic genus $p_a(X) = 0$}
\begin{theorem}\label{thm: Tanaka's classfy P1}
Let $X$ be a normal projective integral curve defined over a field $K$ with $H^0(X,\Oo_X) = K$ and $H^1(X,\Oo_X) = 0$. Then the following statements hold.
\begin{enumerate}
    \item The degree of the canonical divisor $\deg_K K_X = -2$.
    \item The curve $X$ is isomorphic to a conic in $\Pp^2_K$. Moreover, $X \simeq \Pp^1_K$ if and only if it has a $K$-rational point.
    \item Either $X$ is a smooth conic or $X$ is geometrically non-reduced. In the latter case, we have $\chara K = 2$, and $X$ is isomorphic to the curve defined by a quadric $s x^2+t y^2+z^2 =0$ for some $s,t \in K\backslash K^2$.
    \item Assume that $X$ is geometrically non-reduced. Let $P\in X$ be a closed point such that $[k(P):K] =2$. Then $X\times_Kk(P)$ is reduced but not normal. Let $\nu: X'\to X\times_Kk(P)$ be the normalization morphism and $K'=H^0(X',\mathcal{O}_{X'})$. Then $[K':K]= 4$, $X\times_{\mathrm{Spec} K} \mathrm{Spec} K'$ is not reduced and $X' \cong \Pp^1_{K'}$.
     Furthermore, the induced morphism $$X' \rightarrow (X\times_{\mathrm{Spec} K} \mathrm{Spec} K')_{\rm red}^\nu$$
     is an isomorphism.
\end{enumerate}
\end{theorem}
\begin{proof}
Refer to \cite[Theorem 9.10]{Tan21} for the statements (1,2,3). 

(4) Note that $H^0(X\times_K k(P), \Oo_{X\times_K k(P)})= k(P)$, and the curve $X\times_K k(P)$ has a $k(P)$ rational point $Q$, which is over $P$. We conclude that $X\times_Kk(P)$ is reduced by \cite[Lemma 1.3]{Sch10}, but is  not normal at $Q$ because otherwise $X\times_Kk(P)$ will be  geometrically reduced, see \cite[Corollary 2.14]{Tan21}. Let $\nu: X'\to X\times_K k(P)$ be the normalization morphism and $K'=H^0(X',\mathcal{O}_{X'})$. We have the following commutative diagram
\[
\begin{tikzcd}
X' \arrow{d} \arrow{r}{\nu}\arrow[bend right=-30]{rr}{\pi}& X\times_{K} k(P) \arrow{d}\arrow{r}
& X\arrow{d}\\
{\rm Spec} \; K' \arrow{r}& {\rm Spec}\; k(P) \arrow{r}& {\rm Spec} \; K
\end{tikzcd}
\] 
 Let $P'\in X'$ be the point over $Q$, which lies in the support of the conductor. We may write that 
 \[
 K_{X'}\sim \nu^*K_{X\times_Kk(P)} - (P'+ C'),
 \]
 where $C'\geq 0$. Combining $\deg_{K'}K_{X'}=-2$ and $\deg_{K'} \nu^*K_{X\times_Kk(P)} <0$, we conclude that 
 $$-\deg_{K'}\nu^*K_{X\times_Kk(P)} = \deg_{K'}P' = 1.$$
 Hence $C'=0$ and $P'$ is a $K'$ rational point, as a result $X'\cong \Pp^1_{K'}$. Denote by $\pi: X'\to X$ the natural morphism, which has $\deg \pi = 2$. Then by 
$$-1= \deg_{K'}(\nu^*K_{X\times_Kk(P)} =\pi^*K_X)=\frac{\deg_{K} (\pi^*K_X)}{[K':K]} = \frac{-4}{[K':K]} ,$$
it follows that $[K':K]=4$. 

\medskip

Finally by the universal property of the base change $X\times_{\mathrm{Spec} K} \mathrm{Spec} K'$, we have a natural morphism $X'\rightarrow (X\times_{\mathrm{Spec} K} \mathrm{Spec} K')_{\rm red}^\nu$. Since $\pi: X'\to X$ is a finite purely inseparable morphism of degree two, we see that $X\times_{\mathrm{Spec} K} \mathrm{Spec} K'$ is not reduced, and thus $X'\rightarrow (X\times_{\mathrm{Spec} K} \mathrm{Spec} K')_{\rm red}^\nu$ is finite and of degree one, hence it is an isomorphism (see also \cite[Corollary 2.4]{Tan21}).
\end{proof}

\subsection{Trace maps of iterations of absolute Frobenius morphisms}

Let $X$ be a variety over $k$ and $F_X^e: X^{1/p^{e}} \to X$ be the absolute Frobenius morphism of $X$ iterated $e$-times (we also use $X^{1/p} \to X$ to denote the absolute Frobenius morphism). Let $(X, \De)$ be a log pair (i.e., $X$ is a normal variety and $\De \geq 0$ is a $\Qq$-Cartier divisor). Then $(X, \De)$ is called \emph{strongly F-regular} if for every effective Weil divisor $D$ on $X$, there exists an $e\in \Zz_{>0}$ such that the natural morphism
\[
\Oo_X \xrightarrow{F_X^{e \sharp}} F_{X*}^e \Oo_X \hookrightarrow F_{X*}^e\Oo_X(\lceil (p^e-1) \De +D \rceil)
\] splits locally as an $\Oo_X$-module homomorphism. 

A $\Qq$-Weil divisor $D$ is called a \emph{$\Zz_{(p)}$-Cartier divisor} if there exists some $b \in \Zz_{>0}$ not divided by $\chara k=p$ such that $bD$ is Cartier.

Let $X$ be a normal projective variety over $k$ and $E$ be an effective Cartier divisor. Then for every positive integer $e$, there exists a trace map (For example, see \cite[Proposition 2.1]{Tan15})
\[
{\rm Tr}^{e}_{X,E}: F^e_{X*}\omega_X(E) \to \omega_X(E). 
\] 
If $D$ be a Cartier divisor on $X$, then define ${\rm Tr}^e_{X,E}(D)$ to be the map that is tensoring ${\rm Tr}^e_{X,E}$ with $\Oo_X(D)$. When $E=0$, we write ${\rm Tr}^{e}_{X}, {\rm Tr}^{e}_{X}(D)$ instead of ${\rm Tr}^{e}_{X,0}, {\rm Tr}^{e}_{X,0}(D)$ for simplicity. Taking the global sections, ${\rm Tr}^e_X(-K_X+D)$ induces
    \[
    \Phi_e: H^0(X, F_{X*}^e\omega_X\otimes \Oo_X(-K_X+D)) \to H^0(X, \Oo_X(D)).
    \] Set
    \begin{equation}\label{eq: S0}
    S^0(X,D) \coloneqq \bigcap_{e \geq 0} {\rm Im}(\Phi_e).
    \end{equation}     
    Note that \[
  F_{X*}^e\omega_X\otimes \Oo_X(-K_X+D)=F_{X*}^e\Oo_X((1-p^e)K_X+p^eD).
    \]
    Moreover, if $e'>e$, then ${\rm Im}(\Phi_{e'}) \subset {\rm Im}(\Phi_e)$ as 
    \[
    {\rm Tr}^{e'}_X(-K_X+D) =  {\rm Tr}^{e}_X(-K_X+D) \circ \left( F_{X*}^{e}{\rm Tr}^{e'-e}_X(-p^e K_X+p^eD) \right). 
    \]

The following result should be well-known to the experts. We provide a proof due to the lack of a suitable reference.

\begin{proposition}\label{prop: nonzero V0}
    Let $C$ be a normal projective integral curve over a field $K$ with $\chara K =p>0$ ($K$ is not necessarily algebraically closed). Suppose that $H^0(C,\Oo_C) = K$ and $p_a(C)=0$, then $S^0(C, -K_C) \neq \emptyset$. Moreover, if $C\simeq\Pp_K^1$, then $S^0(C, D) \neq \emptyset$ for any Cartier divisor $D$ with $\deg D>0$.
\end{proposition}

\begin{proof}
By Tanaka's classification of curves with $p_a(C)=0$ (Theorem \ref{thm: Tanaka's classfy P1}), $C$ is isomorphic to a curve in $\Pp^2_K$ defined by a quadric $G(X,Y,Z)=0$. Then we have the following commutative diagram  (for example, see \cite[Lemma 2.6]{Tan15}, proof of \cite[Theorem 4.14]{Pat16})
\[
\begin{tikzcd}
0\arrow{r}
& F_*^e\omega_{\Pp^2} \arrow{r}\arrow{d}
&F_*^e(\omega_{\Pp^2}(C)) \arrow{r}{\alpha}\arrow{d}{{\rm Tr}^e_{\Pp^2,C}}
&F_*^e\omega_{C} \arrow{r}\arrow{d}{{\rm Tr}^e_C}
& 0\\
0\arrow{r}
&\omega_{\Pp^2} \arrow{r}
&\omega_{\Pp^2}(C) \arrow{r}{\beta}
&\omega_{C} \arrow{r}
& 0.
\end{tikzcd}
\]
Tensoring with $\Oo_{\Pp^2}(-2(K_{\Pp^2}+C))$ and taking global sections, by the adjunction formula we have

\begin{tikzcd}
0 \arrow[d]                                                                                                                                                    &  & 0 \arrow[d]                                                                            \\
{H^0(\mathbb{P}^2,F^e_*\omega_{\Pp^2}\otimes \mathcal{O}_{\mathbb{P}^2}(-2K_{\mathbb{P}^2}-2C))} \arrow[rr] \arrow[d]                                          &  & {H^0(\mathbb{P}^2,\omega_{\mathbb{P}^2}(-2K_{\mathbb{P}^2}-2C))} =0\arrow[d]           \\
{H^0(\Pp^2,F_*^e\omega_{\Pp^2}(C)\otimes \Oo_{\Pp^2}(-2K_{\Pp^2}-2C))} \arrow[rr, "{{\rm Tr}^e_{\Pp^2, C}(-2K_{\Pp^2}-2C)}"] \arrow[d, "H^0(\alpha\otimes 1)"] &  & {H^0(\Pp^2,\omega_{\Pp^2}(C)\otimes(-2K_{\Pp^2}-2C))} \arrow[d, "H^0(\beta\otimes 1)"] \\
{H^0(C,F_*^e\omega_{C}\otimes\Oo_C(-2K_C))} \arrow[rr, "{\rm Tr}^e_C(-2K_C)"] \arrow[d]                                                                        &  & {H^0(C,\Oo_C(-K_C))} \arrow[d]                                                         \\
{H^1(\mathbb{P}^2,F^e_*\omega_{\Pp^2}\otimes \mathcal{O}_{\mathbb{P}^2}(-2K_{\mathbb{P}^2}-2C))}                                                             &  & {H^1(\mathbb{P}^2,\omega_{\mathbb{P}^2}(-2K_{\mathbb{P}^2}-2C))=0.}                    
\end{tikzcd}\\
The two vanishings in the above diagram hold because $\omega_{\mathbb{P}^2}(-2K_{\mathbb{P}^2}-2C) \sim \Oo_{\Pp^2}(-1)$, and thus $H^0(\beta\otimes 1)$ induces $H^0(\Pp^2,\omega_{\Pp^2}(C)\otimes(-2K_{\Pp^2}-2C)) \cong H^0(C,\Oo_C(-K_C))$.


To see $S^0(X,-K_C) \neq \emptyset$, it suffices to show that ${\rm Tr}^e_{\Pp^2,C}(-2K_{\Pp^2}-2C)$ is nonzero for all $e\geq 1$. Let $[X:Y:Z]$ be a homogeneous coordinate of $\Pp^2$. Set $H = \{Z=0\}$. As $-2K_{\Pp^2}-2C\simeq 2H$, it suffices to show 
\[
{\rm Tr}^e_{\Pp^2,C}(2H):H^0(\Pp^2,\omega_{\Pp^2}(C+2p^eH))\rightarrow H^0(\Pp^2,\omega_{\Pp^2}(C+2H))
\] is nonzero (see \cite[Lemma 2.4]{Tan15}).

Let $U_{2}\coloneqq \{[x:y:1] \in \Pp^2\}$ be an open set. Consider
\[
\gamma \coloneqq \frac{x^{p^e-1}y^{p^e-1}}{G(x,y,1)}~dx \wedge dy \in H^0(U_2, \omega_{\Pp^2}(C+2p^eH)).
\] Then $\gamma$ extends to a section of $H^0(\Pp^2, \omega_{\Pp^2}(C+2p^eH))$. In fact, on the open set $U_{0}\coloneqq \{[1:y:z] \in \Pp^2\}$, we have 
\[
[1/z:y/z:1] \in U_0 \cap U_2.
\] Hence,
\[
\gamma=\frac{(1/z)^{p^e-1}(y/z)^{p^e-1}}{G(1/z,y/z,1)} ~ d(\frac 1 z) \wedge d(\frac y z)= \frac{y^{p^e-1} z^{1-2p^e}}{G(1,y,z)}~dz \wedge dy
\] extends to a section of $H^0(U_0, \omega_{\Pp^2}(C+2p^eH))$. A similar calculation shows that $\gamma$ also extends to a section of $H^0(U_1, \omega_{\Pp^2}(C+2p^eH))$ with $U_1 = \{[x:1:z]\in \Pp^2\}$. Now, by the explicit description of the trace map on the smooth affine set (for example, see \cite[Remark 2.2 and Remark 2.3]{Tan15} or \cite[Lemma 1.3.6]{BK07}), we have
\[
{\rm Tr}^e_{\Pp^2,C}(2H): 
\frac{x^{p^e-1}y^{p^e-1}}{G(x,y,1)}~dx\wedge dy 
\mapsto \frac{1}{G(x,y,1)}~dx\wedge dy. 
\] In particular, ${\rm Tr}^e_{\Pp^2,C}(2H)$ is nonzero, and thus shows the first claim.

\medskip

When $C\simeq \Pp^1_K$, then as it is Frobenius split by \cite[Example 2.2]{SZ14}, there exists an $\Oo_C$-module homomorphism $\phi$ such that composition morphism
\[
\Oo_C \xrightarrow{F_C^{e \sharp}} F_{C*}^e \Oo_C 
\xrightarrow \phi \Oo_C\] is the identity.  Tensoring with $\Oo_C(-D)$ and applying the functor $\Hom(-,\omega_C)$, we get 
\[
{\rm id}:\omega_C(D)\rightarrow F_{C*}^e \omega_C\otimes\Oo_C(D)\xrightarrow {{\rm Tr}_C^e(D)} \omega_C(D).
\]
Tensoring with $\omega_C^{-1}$ and taking global sections, we get
\[
{\rm id}:H^0(C,\Oo_C(D))\rightarrow H^0(C,F_{C*}^e\omega_C\otimes\Oo_C(-K_C+D))\xrightarrow {\Phi_e} H^0(C,\Oo_C(D)).
\]
Therefore, $S^0(C,D) = \bigcap_{e \geq 0} {\rm Im}(\Phi_e)=H^0(C,\Oo_C(D))\neq \emptyset$.
\end{proof}

\section{Some related results and immediate consequences} \label{sec: related results}
In this section, we recall some known results on varieties with nef anticanonical divisors in characteristic $p$, and derive several immediate consequences for further study.

\subsection{Characterization of abelian varieties} 

\begin{theorem}[{\cite[Theorem 0.2]{HPZ19} and \cite[Proposition 2.9]{CWZ23}}]\label{thm: cover of abelian variety}
    Let $X$ be a normal $\Qq$-Gorenstein (i.e., $K_X$ being $\Qq$-Cartier) projective variety with a generically finite morphism $f: X \to A$ to an abelian variety. Assume that $X$ admits a resolution of singularities $\rho: Y \to X$. Then
\begin{enumerate}
    \item $\kappa(X, K_X ) \geq 0$, and the equality is attained if and only if $X$ is birational to an abelian variety;
    \item if moreover $K_X \equiv 0$, then $X$ is isomorphic to an abelian variety.
\end{enumerate} 
\end{theorem}

\subsection{The Albanese morphism of varieties with nef anticanonical divisors}

\begin{theorem}[{\cite[Theorem 1.5, Theorem 8.1]{CWZ23}}]\label{thm: alb is a fibration}
    Let $(X, \De)$ be a projective normal $\Qq$-factorial klt pair with  $-(K_X + \De)$ being nef. Assume the resolution of singularities holds in dimension $< \dim(X)$. If the Albanese morphism $a_X: X \to A$ is of relative dimension one over the image $a_X(X)$. Then $a_X: X \to A$ is a fibration.
\end{theorem}

\begin{remark}
    The assumption on the resolution of singularities is needed since the argument heavily relies on Theorem \ref{thm: cover of abelian variety}.
\end{remark}

\begin{theorem}[{\cite[Theorem 5.7]{Eji23}}]\label{thm: iso fibers}
    Let $(X, \De)$ be a strongly F-regular projective variety and let $Y$ be an abelian variety. Let $f: X \to Y$ be an algebraic fiber space. Suppose that $-K_X-\De$ is a nef $\Zz_{(p)}$-Cartier divisor, and $(X_{\bar\eta}, \De_{\bar\eta})$ is strongly F-regular. Then $X_y \simeq X_z$ for every $y,z \in Y(k)$.
\end{theorem}

\subsection{A semiampleness criterion}

When we deal with a variety whose Albanese morphisms are geometrically non-reduced, we shall reduce it to the following case.

\begin{proposition}[{\cite[Proposition 5.3]{CWZ23}}]\label{prop: cwz23 prop 5.2}
Assume that a variety admits a smooth resolution of singularities. Let $X$ be a $\Qq$-factorial normal projective variety, and let $\De$ be an effective $\Qq$-divisor on $X$. Let $f: X \to S$ be a fibration of relative dimension one, and $\mathfrak M$ a movable linear system on $X$ without fixed components and with $\deg_{K(S)} \mathfrak M > 0$. Assume that the following three conditions hold:
\begin{enumerate}
\item $S$ is of maximal Albanese dimension,
\item  $-(K_X + M_0 + \De)$ is nef, where $M_0 \in \mathfrak M$, and
\item either $(X_{K(S)}, \De_{K(S)})$ is klt, or if $T$ is a (the unique) horizontal irreducible component of $\De$ with coefficient one then $\deg_{K(S)} T = 1$ and the restriction $T|_{T^\nu}$ on the normalization of $T$ is pseudo-effective.
\end{enumerate}
Then
\begin{enumerate}[label=(\roman*)]
\item $S$ is an abelian variety,
\item $M_0$ is semi-ample with numerical dimension $\nu(X, M_0) = 1$, that is, $|M_0|$ defines a fibration $g: X \to \Pp^1$, and
\item for a general $t \in \Pp^1$, the fiber of g over $t$ (denoted by $G_t$) is isomorphic to an abelian variety, and  $\De|_{G_t} \equiv 0$.
\end{enumerate}
\end{proposition}

\subsection{The twisted sections of anticanonical divisor}

By an argument analogous to that in \cite[Section 4.2]{Zha20}, we have the following nonvanishing result (also see \cite[Theorem 3.7]{Zha20}).

\begin{theorem}[{\cite[Theorem 3.6]{CWZ24}}]\label{thm: cwz24}
Let $X$ be a normal projective variety equipped with a surjective morphism $f: X \to A$ to an abelian variety $A$ of dimension $d$. Let $\De$ be an effective divisor on $X$ and $D$ be a Cartier divisor on $X$. Assume that
\begin{enumerate}[label=(\alph*)]
    \item $K_X + \De$ is a $\Qq$-Cartier $\Qq$-divisor whose Weil index of $K_X + \De$ is indivisible by $p$; 
    \item the Cartier index of $(K_X + \De)|_{X_\eta}$ is indivisible by $p$;
    \item $D-(K_X + \De)$ is nef and relatively ample over $A$; 
    \item $r=\dim_{K(\eta)}S^0_\De(X_\eta,D|_{X_\eta})>0$.
\end{enumerate} 
Then
(i) $V^0(f_*\Oo_X(D)) \coloneqq \{\alpha \in \hat A=\Pic^0(A)\mid h^0(f_*\Oo_X(D) \otimes P_\alpha) > 0\} \neq \emptyset$, where $P_\alpha$ denotes the Poincar\'e line bundle over $A \times \hat A$; and

(ii) if $\dim V^0(f_*\Oo_X(D)) = 0$, then there exists a subsheaf $\Ff\subset f_*\Oo_X(D)$ of rank $r$ such that $\Ff|_{X_\eta} = S^0_\De(X_\eta,D|_{X_\eta})$. Moreover, there exist an isogeny of abelian varieties $\pi: A_1 \to A$ and $P_1, \ldots, P_r \in  \Pic^0(A_1)$ inducing a generically surjective homomorphism
\[
\beta: \bigoplus_i P_i \to \pi^*\Ff.
\]
\end{theorem}

\begin{proposition}\label{prop: twisted section}
Let $X$ be a smooth projective variety over an algebraically closed field $k$ of characteristic $p>0$. Assume that 
\begin{enumerate}
\item the resolution of singularities holds in dimension $<\dim X$,
    \item $-K_X$ is nef and relatively semiample over $A$,
    \item the Albanese morphism $a_X: X \to A$ is of relative dimension one, which is a fibration;
    \item $\deg_{K(A)} K_X <0$.
\end{enumerate}
Then 

(i) There exists an $L \in \Pic^0(A)$ such that 
    \begin{equation}\label{eq: nonvanishing twist pic0}
       -K_X+a_X^*L \sim B \geq 0. 
    \end{equation}
    
(ii) The divisor $B$ has no vertical components over $A$. Each irreducible component $T$ of $B$ is a nef and normal divisor, and the projection $T \to A$ is an isogeny of abelian varieties. 

Precisely, we have the following cases:

 {\bf Case (a).} $B=2T$. Then $T \to A$ is an isomorphism, and $T$ is nef with $T|_T\equiv 0$.

    {\bf Case (b).} $B=T+T'$ with $T \neq T'$ and $T\cap T' = \emptyset$. In this case, both $T\rightarrow A$ and $T'\rightarrow A$ are isomorphisms. Furthermore, both $T$ and $T'$ are nef, and $L$ is a torsion bundle.

    {\bf Case (c).} $B=T$. Then $T \to A$ is an isogeny of degree $2$, and T is nef. In this case, $L$ is a torsion bundle.
\end{proposition}

\begin{proof}
Under this situation, the generic fiber is a curve of arithmetic genus zero over $K(A)$. 
By Theorem \ref{thm: cwz24} and Proposition \ref{prop: nonzero V0}, there exists an $L \in \Pic^0(A)$ such that 
    \begin{equation}\label{eq: nonvanishing twist pic0}
       -K_X+a_X^*L \sim B \geq 0. 
    \end{equation} 
Write $B = B^h + B^v$, where $B^h$ and $B^v$ denote the horizontal and vertical parts of $B$ over $A$, respectively.   Let $T$ be an irreducible component of $B^h$. As $\deg_{k(\eta)} B^h =2$, we have the following three cases:

    \medskip

    {\bf Case (a).} $B^h=2T$, hence $T \to A$ is generically of degree $1$;

    {\bf Case (b).} $B^h=T+T'$ with $T \neq T'$, hence $T \to A$ is still generically of degree $1$;

    {\bf Case (c).} $B^h=T$, hence $T \to A$ is generically of degree $2$.
    
  Next, we show the remaining statements case by case. 

In Case (a), applying the adjunction formula (Theorem \ref{thm: adjunction}) on the normalization $T^\nu \to T$ , we have
        \[
        \begin{split}
            a_X^*L|_{T^\nu} &\sim (K_X+T)|_{T^\nu}+T|_{T^\nu}+B^v|_{T^\nu}\\
            &\sim K_{T^\nu}+\De_{T^\nu}+\frac 1 2 (-K_X+a_X^*L-B^\nu)|_{T^\nu}+B^v|_{T^\nu}\\
            & \sim K_{T^\nu}+\De_{T^\nu}+\frac 1 2 (-K_X+a_X^*L)|_{T^\nu}+\frac 1 2 B^v|_{T^\nu}.
        \end{split}
        \] As $T^\nu$ is of maximal Albanese dimension, we have $K_{T^\nu} \succcurlyeq 0$ by Theorem \ref{thm: cover of abelian variety}. Since $\De_{T^\nu} \geq 0, B^v|_{T^\nu} \geq 0$ and $(-K_X+a_X^*L)|_{T^\nu}$ is nef, by $a_X^*L|_{T^\nu} \equiv 0$, we conlcude
        \begin{equation}\label{eq: rest T}
        K_{T^\nu} \equiv \De_{T^\nu}\equiv (-K_X+a_X^*L)|_{T^\nu}\equiv B^v|_{T^\nu}\equiv 0.
        \end{equation} Therefore, $T^\nu$ is an abelian variety and $T^\nu \to A$ is an isomorphism by Theorem \ref{thm: cover of abelian variety}. This implies that $a_X|_T: T \to A$ is an isomorphism and thus $T^{\nu} = T$. As $B^v|_{T^\nu}\equiv 0$, we have $B^v=0$.
        Therefore, $T \sim_{\mathbb{Q}} \frac 1 2 (-K_X+a_X^*L)$ is a nef divisor, and $T|_T\equiv 0$ by \eqref{eq: rest T}. 
 \medskip
        
In Case (b), applying the adjunction formula on the normalization $T^\nu \to T$ again, we have 
         \[
        \begin{split}
            a_X^*L|_{T^\nu} &\sim (K_X+T)|_{T^\nu}+T'|_{T^\nu}+B^v|_{T^\nu}\\
            &\sim K_{T^\nu}+\De_{T^\nu}+T'|_{T^\nu}+B^v|_{T^\nu}.
        \end{split}
        \] Note that we have $a_X^*L|_{T^\nu} \equiv 0$, $K_{T^\nu} \succcurlyeq 0$ (by Theorem \ref{thm: cover of abelian variety}), $\De_{T^\nu} \geq 0, T'|_{T^\nu} \geq 0, B^v|_{T^\nu} \geq 0$. From this we conclude $K_{T^\nu} \equiv \De_{T^\nu} \equiv B^v \equiv T'|_{T^\nu} \equiv 0$ and thus $T^\nu = T \simeq A$. In turn, since $a_X^*L \equiv 0$ and $-K_X$ is nef, we see that $T|_T\equiv - K_X|_T$ is nef, thus  $T$ is nef. The claimed property also holds for $T'$ by the same argument. Besides, as $K_{T^\nu}+\De_{T^\nu}+T'|_{T^\nu}+B^v|_{T^\nu} \succcurlyeq 0$, we have $a_X^*L|_{T^\nu} \succcurlyeq 0$ and thus $L$ is a torsion bundle.

In Case (c), applying the adjunction formula on the normalization $T^\nu \to T$ again, we have 
         \[
        \begin{split}
            a_X^*L|_{T^\nu} &\sim (K_X+T)|_{T^\nu}+B^v|_{T^\nu}\\
            &\sim K_{T^\nu}+\De_{T^\nu}+B^v|_{T^\nu}.
        \end{split}
        \]
         Similarly to the previous cases, we can show that    
          $a_X^*L|_{T^\nu}\equiv 0$, $K_{T^\nu} \equiv \De_{T^\nu} \equiv B^v|_{T^\nu} \equiv 0$, hence $L$ is a torsion bundle,  $T$ is normal and is isomorphic to an abelian variety, thus $a_X|_T: T \to A$ is an isogeny of abelian varieties of degree 2. In turn, since $a_X^*L \equiv 0$ and $-K_X$ is nef, we see that $T|_T\equiv - K_X|_T$ is nef, thus  $T$ is nef. 
\end{proof}

\section{Albanese morphisms with geometrically reduced fibers}\label{sec: geo reduced fibers}

\subsection{Setting}\label{sec:notation1}
Let $X$ be a smooth projective variety over an algebraically closed field $k$ of characteristic $p>0$. Assume that 
\begin{enumerate}
\item the resolution of singularities holds in dimension $<\dim X$,
    \item $-K_X$ is nef,
    \item the Albanese morphism $a_X: X \to A$ is of relative dimension one;
    \item  $a_X: X \to A$  is separable amd the generic fiber has $p_a(X_\eta)$=0.
\end{enumerate}
We have known that $a_X: X \to A$ is a fibration by Theorem \ref{thm: alb is a fibration}.

\subsection{The Albanese morphism $a_X:X\to A$ is smooth}

  \begin{proposition}\label{prop: somooth morph}
The fibration $a_X: X \to A$ is smooth.
  \end{proposition}
  \begin{proof}
      By Theorem \ref{thm: iso fibers}, $a_X$ is equi-dimensional. As $X$ and $A$ are both regular, $a_X$ is flat by the miracle flatness (see \cite[\href{https://stacks.math.columbia.edu/tag/00R4}{Lemma 00R4}]{stacks-project}). And since the generic fiber $X_\eta$ is smooth over $k(\eta)$, by Theorem \ref{thm: iso fibers} we conclude that every fiber of $a_X$ is smooth. Then we can apply the smoothness criterion \cite[\href{https://stacks.math.columbia.edu/tag/01V8}{Lemma 01V8}]{stacks-project} and conclude that $a_X$ is a smooth morphism. 
  \end{proof}

\subsection{Structure of the Albanese morphisms}\label{subsec: structure of alb when separable}
By Theorem \ref{thm: iso fibers}, we have that $-K_X$ is relatively ample over $A$. Then we can apply Proposition \ref{prop: twisted section} and obtain that
there exists an $L \in \Pic^0(A)$ such that 
    \begin{equation}\label{eq: nonvanishing twist pic0}
       -K_X+a_X^*L \sim B \geq 0. 
    \end{equation}
The divisor $B$ falls into the following three cases

 {\bf Case (a).} $B=2T$

    {\bf Case (b).} $B=T+T'$ with $T \neq T'$;

    {\bf Case (c).} $B=T$ and  $T \to A$ is of degree $2$.\\
Each irreducible component $T$ of $B$ is normal, and the projection $T \to A$ is an isogeny of abelian varieties. 
\medskip

We will describe the structure of the variety $X$ in each case. Note that these three cases are not mutually exclusive because of the choices of $L$ and $B$.

\begin{theorem}\label{thm: structure-sep1}  Assume we are in Case (a) or (b) and set $\Ee = (a_X)_*\Oo_X(T)$. Then the following statements hold.
  
  (i) The sheaf $\mathcal{L}:=(a_X|_T)_*\Oo_T(T) \in \mathrm{Pic}^0(A)$, 
  $\Ee$ is a numerically flat vector bundle of rank two on $A$ fitting into the exact sequence
   $$ 0\to \Oo_A \to \Ee \to \mathcal{L} \to 0,$$
    and $X \simeq \Pp_A(\Ee)(:=\mathrm{Proj}_{\mathcal{O}_A}(\bigoplus_{n\geq 0}Sym^n \Ee))$.
  
  (ii) If $\mathcal{L} \neq \mathcal{O}_A$, then $\Ee\cong  \Oo_A \bigoplus \mathcal{L}$; otherwise 
  $\Ee$ corresponds to an element $\alpha\in H^1(A, \mathcal{O}_A)$, and there exists an isogeny $\theta: B \to A$ of abelian varieties such that $\theta^*\Ee =\mathcal{O}_B^{\oplus2}$, thus $X_{B} \simeq B \times \Pp^1$.
   
(iii) $-K_X \sim \mathcal{O}_X(2)\otimes a_X^*\mathcal{L}^{-1}$, and $h^0(X, -K_X) >0$.
\end{theorem}
\begin{proof} 
(i) In Case (a) or (b), by Proposition \ref{prop: twisted section} (ii), $a_X$ admits a section $\sigma: A \to T$. Since $a_X$ is a smooth fibration whose every closed fiber is isomorphic to $\mathbb{P}^1$, it follows that $T$ is very ample on each fiber, and hence relatively ample over $A$. In turn, we conclude that  $\Ee=(a_X)_*\Oo_X(T)$ is a vector bundle of rank two on $A$ by Grauert's Theorem \cite[\Romannum{3}, Corollary 12.9]{Har77}, and thus $X \simeq \Pp_A(\Ee)$ by exactly the same proof of \cite[\Romannum{5}, Proposition 2.2]{Har77}. 

In the following, we show that $\Ee$ is a numerically flat vector bundle on $A$. 
 Applying $(a_X)_*$ to the exact sequence 
\[
0 \to \Oo_X\to \Oo_X(T) \to \Oo_T(T) \to 0,
\] we have 
        \[
        0\to \Oo_A \to \Ee \to (a_X|_T)_*\Oo_T(T)\simeq \sigma^*\Oo_T(T) \to R^1(a_X)_*\Oo_X=0.
        \] 
 Due to Theorem \ref{thm: num flat on abelian} and the above exact sequence, it suffices to show that $\Oo_T(T)$ is numerically trivial.
 \begin{itemize}
     \item In Case (a), this is Proposition \ref{prop: twisted section} (ii) Case (a).
     \item In Case (b), note that $a_X$ is a smooth morphism and $\Oo_X(T-T')|_F\simeq \Oo_F$ for each fiber $F$ of $a_X$. Thus, $\Ll_A:= a_{X*}\Oo_X(T-T')$ is line bundle on $A$ and the natural morphism $a_X^*\Ll_A=a_X^*a_{X*}\Oo_X(T-T')\rightarrow \Oo_X(T-T')$ is isomorphism. Restricting the above line bundle on $T$ and $T'$ respectively, by Proposition \ref{prop: twisted section} (ii) Case (b) we have that $a_X^*\Ll_A|_T=\Oo_X(T-T')|_T=\Oo_X(T)|_T$ is nef, and $a_X^*\Ll_A|_{T'}\cong\Oo_X(T-T')|_{T'}=\Oo_X(-T')|_{T'}$ is anti-nef. It follows that $\Ll_A$ is a numerically trivial line bundle and thus $T\equiv T'$. Therefore, $T|_T\equiv T'|_T=0$ and $T'|_{T'}\equiv T|_{T'}=0$.
 \end{itemize}
      
(ii) The first part follows from ${\rm Ext}^1(\Ll, \Oo_A) =0$ if $\mathcal{L} \neq \mathcal{O}_A$ (see \cite[\S 9, 9.15]{EVB24}). And the second part is a consequence of killing cohomology,  see \cite[Lemma 1.3]{HPZ19}.

(iii) It is well known that $-K_X \sim \mathcal{O}_X(2)\otimes a_X^*\mathcal{L}^{-1}$. Then by $a_{X*}\mathcal{O}_X(-K_X) \cong {\rm Sym}^2 \Ee \otimes \mathcal{L}^{-1}$, it follows that $H^0(X, -K_X) \cong H^0(A,  {\rm Sym}^2 \Ee \otimes \mathcal{L}^{-1}) \neq 0$.

\end{proof}
 
\begin{theorem}\label{thm: structure-sep2}
Assume we are in Case (c). Then the following statements hold.

    (i) The base change  $Y=T \times_A X \simeq \Pp_T(\Gg)$ where $\Gg$ is a numerically flat vector bundle of rank $2$.
    
    (ii) If moreover $T \to A$ is inseparable, then $T$ is semiample, and there exists an isogeny $B \to A$ of abelian varieties such that $X_{B} \simeq B \times \Pp^1$.
    
    (iii) $H^0(X, \Oo_X(-2K_X)) \neq 0$. 
\end{theorem}

\begin{proof}
        In this case, $\tau=(a_X)|_T: T\to A$ is an isogeny of abelian varieties of degree 2. 
        Do the base change $\tau:T\to A$
        \[\xymatrix{&Y=T \times_A X \ar[r]^<<<<{\pi}\ar[d] &X\ar[d]^{a_X}\\
&T\ar[r]_{\tau} &A,}
\]
where the projection $Y \to T$ is also a smooth morphism coinciding with the Albanese morphism.

(i) As $a_X$ is a smooth morphism by Proposition \ref{prop: somooth morph}, the projection $Y \to T$ is also a smooth morphism. Since $T$ is smooth, $Y$ is also smooth (see \cite[Chapter 4, Theorem 3.3]{Liu02}). We have that $K_Y = \pi^*K_X$, and thus $-K_Y$ is nef. 
As a result, $Y$ falls into Case (a) or (b). In turn, the assertion (i) follows from applying Theorem \ref{thm: structure-sep1} (i).

 (ii)       If $\tau$ is inseparable then we can write that $\pi^*T = 2S$, and $S \to T$ is an isomorphism. Applying the adjunction formula, we have that 
 $$(K_X+ T)|_T \sim \mathcal{O}_T~{\rm and~ thus} ~\pi^*(K_X +T)|_S \sim (K_Y+ S+S)|_S \sim \mathcal{O}_S.$$
        From  this we conclude that $S|_S \sim 0$ and thus $S$ is semi-ample on $Y$ by
        \cite[Lemma 4.1]{Tot09b}. 
         Under this situation we have that $Y=\Pp_T((a_Y)_*\mathcal{O}_Y(S))$, and $0 \to \Oo_T \to (a_Y)_*\mathcal{O}_Y(S) \to \Oo_T \to 0$. We can apply Theorem \ref{thm: structure-sep1} (ii) and obtain that there exists an isogeny $B \to A$ of abelian varieties such that $X_{B} \simeq B \times \Pp^1$.

      (iii)  By Theorem \ref{thm: structure-sep1} (iii) we have $H^0(Y, -K_Y) \neq 0$. Since $\pi: Y\to X$ is a finite morphism of degree two, the norm of a nonzero global section of $\Oo_Y(-K_Y)\cong \pi^*\Oo_X(-K_X)$ gives a nonzero section of  $\Oo_X(-2K_X)$.    
     \end{proof}

\section{Albanese morphisms with geometrically non-reduced fibers}\label{sec: geo non-reduced fibers}

\subsection{Setting}\label{sec:notation2}
Let $X$ be a smooth projective variety over an algebraically closed field $k$ of characteristic $p>0$. Assume that 
\begin{enumerate}
\item the resolution of singularities holds in dimension $<\dim X$,
    \item $-K_X$ is nef,
    \item the Albanese morphism $a_X: X \to A$ is of relative dimension one,  which is a fibration by Theorem \ref{thm: alb is a fibration};
    \item  $a_X: X \to A$  is inseparable and the generic fiber has $p_a(X_\eta)$=0.
\end{enumerate}
This situation happens only when $\chara k=2$ by Theorem \ref{thm: Tanaka's classfy P1}. 
In the following, we use $f =a_X: X\to A$.

\subsection{Structure theorem}
By Theorem \ref{thm: Tanaka's classfy P1}, $X_{k(A)^{\frac 1 p}}$ is non-reduced. Let
    \[
    Y\coloneqq (X_{A^{\frac 1 p}})_{\rm red}^\nu
    \] to be the normalization of $(X_{A^{\frac 1 p}})_{\rm red}$.    

\begin{proposition}\label{prop: semiample}
We have $Y\cong  A^{\frac 1 p}\times \Pp^1$, $-K_X$ is semi-ample and $f$-ample.
\end{proposition}
\begin{proof}
Let $\pi: Y \to X$ be the natural morphism. We obtain the following commutative diagram
        \[
\begin{tikzcd}
Y \arrow{rd}{g}\arrow{r}{\nu}\arrow[bend right=-30]{rr}{\pi} & (X_{A^{\frac 1 p}})_{\rm red} \arrow{r} \arrow{d}
& X\arrow{d}{f}\\
 &A^{\frac 1 p}\arrow{r}& A.
\end{tikzcd}
\] 
Applying Theorem \ref{thm: Tanaka's classfy P1}, we can show that the geometric generic fiber of $g:Y \to  A^{\frac 1 p}$ is isomorphic to $\mathbb{P}_{K(A)^{1/p}}^1$, hence $g:Y \to  A^{\frac 1 p}$ is a fibration. 
By results of Section \ref{sec: foliation}, we have 
    \begin{equation}\label{eq: pullback}
    \pi^*K_X \sim K_Y+(p-1) \det \Omega_{Y/X}^1.
    \end{equation}
Since $f:X \to A$ is inseparable, and $\Omega^1_{A^{\frac 1 p}/A}$ is globally generated, applying Proposition \ref{prop: CWZ23prop3.3}  shows that the linear system of $|\det \Omega_{Y/X}^1|$ contains horizontal movable part. We may write that    
    \[
    |\det \Omega_{Y/X}^1| = \mathfrak M +F,
    \] where $F$ is the fixed part and $\mathfrak M$ is a non-empty movable system. 
To summarize, we have that the pull-back of the anti-canonical divisor is like
$$-\pi^*K_X \sim -(K_Y + \mathfrak M +F).$$
By Proposition \ref{prop: CWZ23prop3.3} and Theorem \ref{thm: Tanaka's classfy P1} $(3)$, we have $\deg_{K(A)^{1/p}} F=0$ and $ \deg_{K(A)^{1/p}}\mathfrak M>0$. Note that $\deg_{K(A)^{1/p}}(-K_Y) = 2$ and $\deg_{K(A)^{1/p}}(-\pi^*K_X) > 0$. From this we can conclude that $\deg_{K(A)^{1/p}}\mathfrak M = \deg_{K(A)^{1/p}}(-\pi^*K_X) = 1$. Applying  Proposition \ref{prop: cwz23 prop 5.2}, the linear system $\mathfrak M$ has no base point and induces a morphism $h: Y \to \mathbb{P}^1$. Each closed fiber $Y_t$ of $h$ is an element of $\mathfrak M$, and the projection $Y_t\to X\cong A^{\frac 1 p}$ is an isomorphism. We can conclude that $(g,h): Y \to  A^{\frac 1 p}\times \Pp^1$ is an isomorphism. Moreover, Proposition \ref{prop: cwz23 prop 5.2} (iii) implies that $F|_{Y_t}\equiv 0$ for general fiber of $h$ over $t\in\Pp^1$. Combining with $\deg _{K(A)^{1/p} }F=0$, the fixed part $F$ is effective vertical divisor of $g$ and $h$ simultaneously, and then $F=0$.

 Thus, $-\pi^*K_X = -(K_Y + \mathfrak M)$ is semi-ample and $g$-ample, which implies that $-K_X$ is semi-ample and $f$-ample.
\end{proof}

\begin{theorem}\label{thm: structure-inseparable}
In the situation of Setting \ref{sec:notation2}, the following statements hold.

(i) There exists an isogeny $A' \to A$ of abelian varieties, which is purely inseparable of degree $4$, such that $X \times_A A'$ is not reduced, and the normalization of $(X \times_A A')_{\mathrm{red}}$ is isomorphic to $A' \times \mathbb{P}_k^1$.
    
(ii) There is an isomorphism $X \cong (A' \times \mathbb{P}_k^1)/\mathcal{F}$, where $\mathcal{F}$ is a smooth rank one foliation on $A' \times \mathbb{P}_k^1$ such that $\mathcal{F} \sim h^*\mathcal{O}_{\mathbb{P}^1}(-1)$, here $h: A' \times \mathbb{P}_k^1 \to \mathbb{P}_k^1$ is the projection onto the second factor.
    
(iii) $h^0\big(X, \mathcal{O}_X(-2K_X)\big) > 0$.
\end{theorem}

\begin{proof}
By Proposition~\ref{prop: semiample}, $-K_X$ is semi-ample and $f$-ample. Thus we can apply Proposition~\ref{prop: twisted section} to find an $L \in \mathrm{Pic}^0(A)$ such that $H^0(X, -K_X + f^*L) \neq 0$. And we can assume that $-K_X + f^*L \sim T$, where $T$ is an effective divisor with $\deg_{K(A)} T = 2$, $T$ is normal, and the projection $T \to A$ is an isogeny of abelian varieties.

Consider the base change $T \to A$. By Theorem~\ref{thm: Tanaka's classfy P1} (iv), $X_T$ is reduced. Let $\nu \colon X' \to X_T$ be the normalization morphism, and let $X' \to B \to T$ be the Stein factorization. This yields the following commutative diagram:
\[
\begin{tikzcd}
X' \arrow{d}{f'}\arrow{r}{\nu}\arrow[bend right=-30]{rr}{\pi}& X_{T} \arrow{d}\arrow{r}
& X\arrow{d}{f}\\
B \arrow{r}&T\arrow{r}& A
\end{tikzcd}
\]
It follows from Theorem~\ref{thm: Tanaka's classfy P1} that $[K(B) : K(A)] = 4$ and $X\times_A B$ is not reduced. We can write that 
\begin{align*}
    K_{X'} &\sim \nu^*K_{X_T} - C - V \\
    & = \pi^*K_X - C - V,
\end{align*}
where $V$ is the vertical component and $C$ is the unique horizontal irreducible component with $\deg_{K(B)} C = 1$. Noting that $X^{1/p} \to X$ factors through $X' \to X$, it follows that $B \to A$ is a finite, purely inseparable morphism of height one and degree $4$. Since $X' \to X$ has degree $2$, there exists a rank one foliation $\mathcal{F}$ such that $X \cong X'/\mathcal{F}$.

\medskip

(i) We first show that $B$ is isomorphic to an abelian variety. Let $C^\nu \to C$ be the normalization morphism of $C$. By the adjunction formula (Theorem~\ref{thm: adjunction}),
\[
- \pi^*K_X|_{C^\nu} \sim - (K_{X'} + C + V) |_{C^\nu} \sim - (K_{C^\nu} + \Delta_{C^\nu} + V|_{C^\nu}).
\]
Since $C^\nu$ has maximal Albanese dimension, $K_{C^\nu} \geq 0$ and both $\Delta_{C^\nu}$ and $V|_{C^\nu}$ are effective. The nefness of $- \pi^*K_X$ then gives
\[
K_{C^\nu} \equiv \Delta_{C^\nu} \equiv V|_{C^\nu} \equiv 0.
\]
Thus $C$ is normal and isomorphic to an abelian variety $A'$, hence $B \cong A'$. 
\smallskip

Next, we show that $X' \cong A' \times \mathbb{P}^1$. As in Proposition~\ref{prop: semiample}, consider the map $f'^*\Omega^1_{A'/A} \to \Omega^1_{X'/X}$. Since $\mathrm{rank}(\Omega^1_{A'/A}) = 2 > \mathrm{rank}(\Omega^1_{X'/X}) = 1$, and $\Omega^1_{A'/A}$ is globally generated, applying Proposition \ref{prop: CWZ23prop3.3}  shows that the linear system of $|\det \Omega_{Y/X}^1|$ contains horizontal movable part. We can write that
\[
|\det \Omega^1_{X'/X}| = \mathfrak M + F,
\]
where $F$ is the fixed part and $\mathfrak M$ the movable part. The pullback of the anti-canonical divisor is
\[
- \pi^* K_X \sim - (K_{X'} + \mathfrak M + F)
\]
and is nef. With $\deg_{K(A')} K_{X'} = -2$ and $\deg_{K(A')} K_X < 0$, we find $\deg_{K(A')}\mathfrak M = 1$. By Proposition~\ref{prop: cwz23 prop 5.2}, $\mathfrak M$ is base-point free and induces a morphism $h: X' \to \mathbb{P}^1$. Furthermore, for any closed fiber $X'_t$ of $h$, the projection $X'_t \to A'$ is an isomorphism. Thus $(f', h): X' \to A' \times \mathbb{P}^1$ is an isomorphism.
\smallskip

Finally, by Theorem~\ref{thm: Tanaka's classfy P1}, we know that $X'$ coincides with the normalization of  $(X_{A'})_{\mathrm{red}}$. This finishes the proof of (i).

\medskip

(ii) Fix a closed point $t \in \mathbb{P}^1$ and consider the restriction of the divisor $- \pi^* K_X|_{X'_t}$ on the fiber $X'_t \cong A'$
\[
- \pi^* K_X|_{X'_t} \sim - (K_{X'} + X'_t+ F)|_{X'_t} \sim - (K_{X'_t} + F|_{X'_t}) \sim  -F|_{X'_t}.
\]
The nefness of $ -F|_{X'_t}$ implies that $ -F|_{X'_t}=0$. From this, we conclude that
 $F = 0$. As a result
\[
\det \mathcal{F} \sim -\det \Omega_{X'/X}^1 \sim -\mathfrak M \sim h^*\mathcal{O}_{\mathbb{P}^1}(-1).
\]
And since both $X'$ and $X$ are smooth, the foliation $\mathcal{F}$ is smooth (\cite[p.142]{Eke87}). 

\medskip

(iii) Note that $- \pi^* K_X \sim - K_{X'} - \mathfrak M \sim h^* \mathcal{O}_{\mathbb{P}^1}(1)$, so $H^0(X', - \pi^* K_X) \neq 0$. Because $\pi: X' \to X$ is a finite morphism of degree $2$, the norm of a nonzero global section of $\pi^* \mathcal{O}_X(-K_X)$ gives a nonzero section of $\mathcal{O}_X(-2K_X)$.

\end{proof}

\subsection{A full description of $X$}

In this section, we will give a complete description of the foliation $\Ff$. By Theorem \ref {thm: structure-inseparable}, we have the following commutative diagram
\[
\begin{tikzcd}
X'= A'\times \Pp^1 \arrow{d}{f'}\arrow{r}{\pi} & X\cong (A'\times \Pp^1)/\Ff\arrow{d}{f}\\
A' \arrow{r}{\tau} &A
\end{tikzcd}
\] 
where $\tau: A'\to A$ is a purely inseparable isogeny of abelian varieties of degree $p^2=4$. The kernel $\mathrm{ker}(\tau)$ is an infinitesimal group of length $p^2$, we may assume its Lie algebra is spanned by $\alpha, \beta \in \mathrm{ker}(H^0(A', T_{A'}) \to \tau^*H^0(A, T_{A}))$, which satisfies that $[\alpha, \beta]=0$ and $\alpha^2, \beta^2\in \mathrm{Span}\{\alpha,\beta\}$. We may choose
$\alpha,\beta$, which fall into one of the following cases:
\begin{enumerate}
  \item[(I)] $\alpha^2= \beta^2=0$;
  \item[(II)] $\alpha^2 = \alpha$ and $\beta^2=\beta$;
  \item[(III)] $\alpha^2 = \alpha$ and $\beta^2=0$;
  \item[(IV)] $\alpha^2 = \beta$ and $\beta^2=0$.
\end{enumerate}

Let $\dim A' = d$. We may write that $T_{A'}=\Oo_{A'}\cdot\alpha\oplus \Oo_{A'}\cdot\beta \bigoplus_{i=1}^{d-2} \Oo_{A'}\cdot\gamma_i$. Under the isomorphism $T_{X'} \cong f'^*T_{A'}\oplus h^*T_{\Pp^1}$, we identify $\alpha, \beta$ as sections of $T_{X'}$. We consider the subsheaf
$$T_1=\Oo_{X'}\cdot\alpha\oplus \Oo_{X'}\cdot\beta \oplus h^*T_{\Pp^1} \subseteq T_{X'}.$$
By the construction, we have that the foliation $\Ff\subset T_1$.
Since $\Ff\sim h^*\Oo_{\Pp^1}(-1)$, the inclusion $\Ff\subset T_1$ is determined by a non-zero element (unique up to scaling)
$$\delta_{\Ff} \in \Hom(h^*\Oo_{\Pp^1}(-1),T_1)\simeq H^0(A'\times \Pp^1, T_1\otimes h^*\Oo_{\Pp^1}(1)).$$
Let $V_0 = \Pp^1\backslash\{\infty\} \cong \mathbb{A}^1_{(t)}, V_1 = \Pp^1\backslash\{0\}\cong \mathbb{A}^1_{(s)}$ where $t\in V_0, s\in V_1$ are coordinates on each piece such that $t=1/s$ on $V_0\cap V_1$. Identify $h^*\Oo_{\Pp^1}(1))|_{A'\times V_0} =\Oo_{X'}\cdot 1$ and $h^*\Oo_{\Pp^1}(1))|_{A'\times V_1} =\Oo_{X'}\cdot \frac{1}{s}$. Restricting on $A'\times V_0$, under the identification
$T_1\otimes h^*\Oo_{\Pp^1}(1))|_{A'\times V_0} = \Oo_{X'}\cdot\alpha\oplus \Oo_{X'}\cdot\beta \oplus \Oo_{X'} \cdot \partial_t$
the element $\delta_{\Ff}$, over the subset $A'\times V_0$, can be written as
$$\delta=a(t) \alpha+b(t) \beta+c(t)\partial_t.$$
Here $a(t), b(t), c(t)\in k[t]$ satisfy the following conditions
    \begin{itemize}
        \item[(C1)] ${\rm gcd}(a, b, c)=1$ since $\Ff$ is saturated in $T_1$;
        \item[(C2)] by $\Ff \sim\Oo_{\Pp^1}(-1)$ and $\partial_t = s^2\partial_s$, we have $\deg a(t), \deg b(t) \leq 1$ and $\deg c(t) \leq 3$ and at least one equality holds;
             \item[(C3)] $\delta^2=(a(t) \alpha+b(t) \beta+c(t)\partial_t)^2 \in \mathrm{Span}_k (\delta)$ by $\Ff^p \subset \Ff$.
        \end{itemize}
Remark that $[\Ff, \Ff] \subseteq \Ff$ automatically holds since rank $\Ff = 1$. Conversely, every $\delta$ subjected to Conditions (C1,C2,C3) determines a foliation $\Ff$ as required. 

\begin{theorem}\label{thm: classfy F}
With the notation above, $\delta$ is given by one of the following forms. 
\begin{enumerate}[label={\rm (\Roman*)}]
    \item $c=0$, and one of polynomials $a,b$ is constant while the other has degree 1. 
    
    \item $ab|c$ with $0\leq \min\{\deg a,\deg b\} \leq \max\{\deg a, \deg b\}=1$. 
    More precisely, there are mutually different $t_0,t_1,t_2\in k$, such that one of the following holds:
    
    \begin{enumerate}[label={\rm (II-\roman*)}]
        \item $a=1,\; b=\frac{t-t_1}{t_1+t_2},\; c=\frac{(t-t_1)(t-t_2)}{t_1+t_2};$
        \item $a=\frac{t-t_2}{t_1+t_2},\; b=1,\; c=\frac{(t-t_1)(t-t_2)}{t_1+t_2};$
        \item $a=(t-t_2),\; b=(t-t_1),\; c=(t-t_1)(t-t_2);$
        \item $a=(t_0+t_1)(t-t_2),\; b=(t_0+t_2)(t-t_1),\; c=(t-t_0)(t-t_1)(t-t_2).$
    \end{enumerate}
    
    \item $a^2b|c$. More precisely, there are two different elements $t_1,t_2\in k$ and nonzero $s\in k$, such that one of the following holds:
    \begin{enumerate}[label={\rm (III-\roman*)}]
    \item $a=s(t-t_1),\; b=1,\; c=s(t-t_1)^2;$
    \item $a=1,\; b=\frac{t-t_1}{s},\; c=t-t_1;$
    \item $a=(t_1+t_2)(t-t_2),\; b=\frac{t-t_1}{s},\; c=(t-t_1)^2(t-t_2).$
    \end{enumerate}
    
    \item $a^3|c$. More precisely, there are  $t_1, t_2\in k$ and nonzero $s_1, s_2, r_2\in k$, such that one of the following holds:

    \begin{enumerate}[label={\rm (IV-\roman*)}]
    \item $a=1,\; b=\frac{t}{s_1}+t_2,\; c=s_1;$
    \item $a=t,\; b=\frac{1}{s_1}+t_2t,\; c=s_1t^3;$
    \item $a=t+s_1r_2,\; b=\frac{1}{s_1s_2}, \; c=s_1 s_2 (t+s_1r_2)^3;$
    \item $a=t-t_1,\; b=\frac{t-t_2}{s_1s_2(t_1+t_2)}, \;  c=s_1 s_2 (t-t_1)^3.$
    \end{enumerate}

\end{enumerate}
\end{theorem}

\begin{proof}
By $\partial_t^2=0$, $[\alpha, \partial_t]=[\beta, \partial_t]=0$ and $\alpha(r(t))=\beta(r(t))=0$ for any regular function $r(t)$ on $U_0$, we have 
\begin{equation}\label{eq: f^p}
\delta^2=(a(t) \alpha+b(t) \beta+c(t)\partial_t)^2=a^2\alpha^2+b^2\beta^2+ca'\alpha+cb'\beta+cc'\partial_t
\end{equation} where $a', b' ,c'$ are the derivatives of $a(t), b(t), c(t)$ respectively. Due to Conditions (C1) and (C2), any two of $a,b,c$ can not be zero simultaneously.

\medskip

\textbf{Case {\rm (I)}:}

By \eqref{eq: f^p}, Condition (C3) becomes $\frac{ca'}{a}=\frac{cb'}{b}=\frac{cc'}{c}$ whenever the denominators are non-zero.
If $c =0$, then $a=1,\, \deg b=1$, or $b=1,\, \deg a=1$.

In the following, we assume that $c\neq0$. Suppose $a\neq 0$, we have $(\frac{c}{a})'=\frac{ac'-ca'}{a^2}=0$. Hence, $\frac c a = r$ where $r$ is a rational polynomial of $t^2$. Thus, there exists coprime polynomials $\ti f_1$ and $\ti g_1$, such that $r=\frac{\ti f_1^2}{\ti g_1^2}$. Therefore, there exists a polynomial $\ti h_1$ such that $a= {\ti g_1}^2 \ti h_1$ and $c =\ti f^2_1\ti h_1.$
Moreover, $\deg a(t)\leq 1$ implies that $\ti g_1$ is constant. Let $f_1:=\frac{\ti f_1}{\ti g_1}$ and $h_1:={\ti g_1}^2\ti h_1$. Therefore, 
\[
a=h_1, \quad c=f_1^2 h_1=f_1^2a.
\]
In the same way, if $b\neq 0$, we get
\[
b=h_2, \quad c=f_2^2h_2=f_2^2b.
\]
Note that $\rm gcd(a,b,c)=gcd(a,b)=1$. If $a=0$, then $b=1,\, c=f_2^2b=f_2^2$, which contradicts Condition (C2). Since $a$ and $b$ have symmetric positions, we can also exclude the case $b=0$. 

From now on, we assume that $a,b\neq0$. Note that Equation $f_1^2a=c=f_2^2b$ implies that $\deg a= \deg b$. 

If $\deg a=\deg b=0$, then $\deg c=2\deg f_1$ is an even number and hence cannot be 3, which contradicts Condition (C2). 

If $\deg a=\deg b=1$ with $\rm gcd(a,b)=1$, then $a| f_2$ and $b| f_1$. Thus $a$ occurs odd times in $f_1^2 a$ while even times in $f_2^2 b$, which contradicts $f_1^2a=c=f_2^2b$.

\medskip

\textbf{Case {\rm (II)}:}

By \eqref{eq: f^p}, Condition (C3) becomes $\frac{a^2+ca'}{a}=\frac{b^2+cb'}{b}=\frac{cc'}{c}$ whenever the denominators are non-zero. If $c =0$, then $a=b$, which contradicts Condition (C2). 

In the following, we assume that $c\neq 0$. Assume that $a\neq 0$, then $(\frac c a)'= \frac{c'a+ca'}{a^2}=1.$ Hence, $\frac c a = t + \frac{\ti f_1^2}{\ti g_1^2}=\frac{t\ti g_1^2+\ti f_1^2}{\ti g_1^2}$, where $\ti f_1$ and $\ti g_1$ are coprime polynomials. Note that $t\ti g_1^2+ \ti f_1^2$ and $\ti g_1^2$ are coprime polynomials. Thus, there exists a polynomial $\ti h_1$ such that $a= {\ti g_1}^2 \ti h_1$ and $c =(t{\ti g_1}^2+\ti f^2_1)\ti h_1.$ Moreover, $\deg a(t)\leq 1$ implies that $\ti g_1$ is constant. Let $f_1:=\frac{\ti f_1}{\ti g_1}$ and $h_1:={\ti g_1}^2\ti h_1$. Therefore, 
\[
a=h_1, \quad c=(t+f_1^2)h_1=(t+f_1^2)a.
\]
Similarly, if $b\neq 0$, we have
\[
b=h_2, \quad c =(t+ f^2_2)h_2 = (t+ f^2_2)b.
\]
Note that $\rm gcd(a,b)=gcd(a,b,c)=1$. If $a=0$, then $b=1, c=t+f_2^2$, which contradicts Condition (C2). Since $a$ and $b$ have symmetric positions, we can also exclude the case $b=0$. 

From now on, we assume that $a,b\neq0$. The following discussion is based on the Equation 
\begin{equation}\label{eq: case 2}
   (t+f_1^2)a=c=(t+f_2^2)b.  
\end{equation}

If $\deg a=\deg b=0$, then $\deg c=\deg (t+f_1^2)\neq 3$, which contradicts Condition (C2).

\textbf{Case {\rm (II-i)}:} $\deg a=0,\, \deg b=1$

Note that $a=1$ and $b$ has a root $t_1\in k$. Equation \eqref{eq: case 2} implies that $\deg f_1=1$ and $\deg f_2=0$. Denote $t_2:=f_2^2\in k$ and then Equation \eqref{eq: case 2} becomes $t+f_1^2=(t-t_2)b$. Note that the polynomial $t+f_1^2$ has distinct roots $t_1,\, t_2$ since it is not a perfect square. More precisely, we have 
$t+f_1^2=\frac{(t-t_1)(t-t_2)}{t_1+t_2}$
and then
\[
a=1,\quad b=\frac{t-t_1}{t_1+t_2},\quad c=\frac{(t-t_1)(t-t_2)}{t_1+t_2}.
\]

\textbf{Case {\rm (II-ii)}:} $\deg a=1,\, \deg b=0$

Since $a$ and $b$ have symmetric positions, in this case, we have
\[
a=\frac{t-t_2}{t_1+t_2},\quad b=1,\quad c=\frac{(t-t_1)(t-t_2)}{t_1+t_2}.
\]

\textbf{Case {\rm (II-iii\&iv)}:} $\deg a=\deg b=1$ with $\gcd(a,b)=1$.  

\textbf{Case {\rm (II-iii)}:} $\deg f_1=\deg f_2=0$

Denote $t_1:=f_1^2\in k, \, t_2:=f_2^2\in k$. By Equation \eqref{eq: case 2} and $\gcd(a,b)=1$, we have $a|(t-t_2)$ and $b|(t-t_1)$. More precisely, 
\[
a=(t-t_2),\quad b=(t-t_1),\quad c=(t-t_1)(t-t_2).
\]

\textbf{Case {\rm (II-iv)}:} $\deg f_1=\deg f_2=1$

Both polynomials $t + f_1^2$ and $t + f_2^2$ have distinct roots, since they are not perfect squares. Moreover, Equation \eqref{eq: case 2} implies that polynomials $t+f_1^2$ and $t+f_2^2$ have only one common root $t_0\in k$. More precisely, there exist distinct $t_1,\, t_2\in k$, such that
$t+f_1^2=\frac{(t-t_0)(t-t_1)}{t_0+t_1}$ and $t+f_2^2=\frac{(t-t_0)(t-t_2)}{t_0+t_2}$.  
By Equation \eqref{eq: case 2} and $\gcd(a,b)=1$, we have $a|(t-t_2)$ and $b|(t-t_1)$. More precisely, 
\[
a=(t_0+t_1)(t-t_2),\quad b=(t_0+t_2)(t-t_1), \quad c=(t+f_1^2)a=(t-t_0)(t-t_1)(t-t_2). 
\]

\medskip

\textbf{Case {\rm (III)}:}

By \eqref{eq: f^p} Condition (C3) becomes $\frac{a^2+ca'}{a}=\frac{cb'}{b}=\frac{cc'}{c}$ whenever the denominators are non-zero. The case $c=0$ can be ruled out by Conditions (C1) and (C2). In the following, we assume that $c\neq 0$. 
After the same calculations as in case {\rm (I)} and case {\rm (II)}, if $a\neq0$, then $c=(t+f_1^2)a$. If $b\neq 0$, then $c=f_2^2b$. Note that $\rm gcd(a,b)=gcd(a,b,c)=1$. If $a=0$, then $b=1, c=f_2^2b=f_2^2$, which contradicts Condition (C2). If $b=0$, then $a=1, c=(t+f_1^2)a=t+f_1^2$, which contradicts Condition (C2). 

From now on, we assume that $a,b\neq0$. The following discussion is based on the Equation 
\begin{equation}\label{eq: case 3}
   (t+f_1^2)a=c=f_2^2b.  
\end{equation}

\textbf{Case {\rm (III-i)}:} $\deg a=1,\, \deg b=0$

By Equation \eqref{eq: case 3} and $\gcd(a,b)=1$, we have $b=1$, $\deg f_1=0$ and $\deg f_2=1$. Denote $t_1:=f_1^2\in k$, and then Equation \eqref{eq: case 3} becomes $(t-t_1)a= (t+f_1^2)a=c=f_2^2$. Therefore, there exists nonzero $s\in k$, such that
\[
a=s(t-t_1),\quad b=1,\quad c=(t-t_1)a=s(t-t_1)^2.
\]

\textbf{Case {\rm (III-ii)}:} $\deg a=0,\, \deg b=1$
 
By Equation \eqref{eq: case 3} and $\gcd(a,b)=1$, we have $a=1$ and $\deg f_1=\deg f_2=0$. Denote $t_1:=f_1^2\in k$ and nonzero $s:=f_2^2\in k$ and then Equation \eqref{eq: case 3} implies that
\[
a=1,\quad b=\frac{t-t_1}{s},\quad c=t-t_1.
\]

\textbf{Case {\rm (III-iii)}:} $\deg a=1,\, \deg b=1$

Let $t_1,\, t_2$ be root of $b,\, a$ respectively. By Equation \eqref{eq: case 3} and $\gcd(a,b)=1$, we have $b|(t+f_1^2)$, $a|f_2^2$ and then $\deg f_1=\deg f_2=1$, since $\deg c\leq 3$. Moreover, polynomial $t+f_1^2$ has roots $t_1, t_2$ and $f_2$ has root $t_2$ of multiplicity 2. More precisely,  $t+f_1^2=\frac{(t-t_1)(t-t_2)}{t_1+t_2}$, and there exists nonzero $s\in k$, such that $f_2^2=s(t-t_2)^2$. Therefore, Equation \eqref{eq: case 3} implies that 
\[
a=(t_1+t_2)(t-t_2),\quad b=\frac{t-t_1}{s},\quad c=(t-t_1)^2(t-t_2).
\]

\medskip

\textbf{Case {\rm (IV)}:}

By \eqref{eq: f^p} Condition (C3) becomes $\frac{ca'}{a}=\frac{a^2+cb'}{b}=\frac{cc'}{c}$ whenever the denominators are non-zero. If $c=0$, then $a=0$, $b=1$, which contradicts Condition (C2). In the following, we assume that $c\neq 0$, and then the above Equation becomes 
\begin{equation}\label{eq: case 4}
ca'=ac',\quad a^2+cb'=c'b.
\end{equation}
If $a=0$, then $b\neq0$. By calculations in case {\rm (I)} for the second Equation of \eqref{eq: case 4}, we have $c=f_2^2 b$. Thus, we cannot find $b$ and $c$ satisfying both (C1) and (C2) simultaneously. 
If $b=0$, then $a\neq0$. By calculations in case {\rm (I)} for the first Equation of \eqref{eq: case 4}, we have $c=f_1^2 a$. Thus, we cannot find $a$ and $c$ satisfying both (C1) and (C2) simultaneously.
In the following, we assume that $a,b\neq 0$. By calculations in case {\rm (I)}, the first Equation of \eqref{eq: case 4} can be solved and then the second Equation of \eqref{eq: case 4} can be simplified as
\[
c=f_1^2a,\quad (\frac{b}{a})'=\frac{(ab)'}{a^2}=\frac{1}{a^2}\cdot(bc\cdot\frac{1}{f_1^2})'=\frac{(bc)'}{a^2f_1^2}=\frac{1}{f_1^2}.
\]
Solve the differential Equation above using integration by parts, and then there are coprime polynomials $\ti f_2,\ti g_2$, such that 
\begin{equation}\label{eq: case 4'}
  c=f_1^2a, \quad b=\frac{t}{f_1^2}a+\frac{\ti f_2^2}{\ti g_2^2}a=\frac{t\ti g_2^2+f_1^2\ti f_2^2}{f_1^2\ti g_2^2}a.
\end{equation}
The following disscussion is based on the above Equation \eqref{eq: case 4'}. Note that to make $c$ a polynomial of $\deg\leq 3$, it follows that $\deg f_1\leq 1$. 

\textbf{Case {\rm (IV-i)}:} $\deg f_1=0$

In this case, note that $t\ti g_2^2+f_1^2\ti f_2^2$ and $f_1^2\ti g_2^2$ are coprime polynomials. By Equation \eqref{eq: case 4'}, there is a polynomial $\ti h_2$, such that $a=f_1^2\ti g_2^2\ti h_2,\, b=(t\ti g_2^2+f_1^2\ti f_2^2)\ti h_2.$
Then $\deg a\leq 1$ implies that $\deg \ti g_2=0$. Let $f_2:=\ti f_2/\ti g_2,\, h_2:=f_1^2\ti g_2^2\ti h_2$. We get $a=h_2,\, b=(t/{f_1^2}+\ti f_2^2/\ti g_2^2)f_1^2\ti g_2^2\ti h_2=(t/{f_1^2}+f_2^2)h_2,\, c=f_1^2a=f_1^2h_2.$
Note that $h_2=\gcd (a,b,c)=1$ and $\deg f_2=0$ since $\deg b\leq 1$. Denote $s_1:=f_1^2\neq 0,\, t_2:=f_2^2\in k$. Therefore, 
\[
a=1,\quad b=\frac{t}{s_1}+t_2,\quad c=s_1.
\]

\textbf{Case {\rm (IV-ii)}:} $\deg f_1=1,\, \deg \ti g_2=0$ or $\deg f_1=1,\,\ti f_2=0$

By Equation \eqref{eq: case 4'}, to make $b$ a polynomial of $\deg\leq 1$, it follows that $f_1^2|ta$ and $\deg \ti f_2=0$. Since $\deg a\leq 1$, we can assume that $a=t$ and $f_1^2=s_1t^2$ for some nonzero $s_1\in k$. Denote $t_2:=\ti f_2^2/\ti g_2^2\in k$. Therefore, Equation \eqref{eq: case 4'} becomes 
$$a=t,\quad b=\frac{t}{f_1^2}a+\frac{\ti f_2^2}{\ti g_2^2}a=\frac{1}{s_1}+t_2t,\quad c=f_1^2a=s_1t^3.$$

\textbf{Case {\rm (IV-iii\&iv)}:} $\deg f_1=1,\, \ti f_2\neq 0$

By Equation \eqref{eq: case 4'}, to make $b$ a polynomial of $\deg\leq 1$, it follows that $\ti g_2^2|(t\ti g_2^2+f_1^2\ti f_2^2)a$. Because $\ti f_2\neq 0$ and $\gcd(\ti f_2,\ti g_2)=1$, we have $\ti g_2^2|f_1^2a$. Hence, $2\deg \ti g_2\leq 2\deg f_1+\deg a\leq 3 $ and then $\deg \ti g_2=1$.
Thus, $f_1^2=s_1 \ti g_2^2$ for some nonzero $s_1\in k$. Then Equation \eqref{eq: case 4'} can be simplified as 
\begin{equation}\label{eq: case 4.3 & 4.4}
    c=f_1^2a,\quad b=\frac{t+s_1 \ti f_2^2}{s_1 \ti g_2^2}a.
\end{equation}
By Equation \eqref{eq: case 4.3 & 4.4}, to make $b$ a polynomial of $\deg\leq 1$, it follows that $\deg \ti f_2\leq 1$.

\textbf{Case {\rm (IV-iii)}:} $\deg \ti g_2=1,\, \deg \ti f_2=0$

In this case, to make $b$ a polynomial, Equation \eqref{eq: case 4.3 & 4.4} implies that $\ti g_2|(t+s_1 \ti f_2^2)$, $\ti g_2| a$ and hence $b$ is nonzero constant. Denote nonzero $r_2:=\ti f_2^2\in k$. We can assume that $a=t+s_1r_2$ and $\ti g_2^2=s_2(t+s_1r_2)^2$ for some nonzero $s_2\in k$. Therefore, Equation \eqref{eq: case 4.3 & 4.4} becomes 
$$a=t+s_1r_2,\quad b=\frac{t+s_1 \ti f_2^2}{s_1 \ti g_2^2}a=\frac{1}{s_1s_2}, \quad c=f_1^2a=s_1\ti g_2^2a=s_1 s_2 (t+s_1r_2)^3.$$

\textbf{Case {\rm (IV-iv)}:} $\deg \ti g_2=1, \, \deg \ti f_2=1$

Note that the polynomial $t+s_1\ti f_2^2$ has two distinct roots $t_1, t_2\in k$, since it is not a perfect square.  More precisely,  $t+s_1\ti f_2^2=\frac{(t-t_1)(t-t_2)}{t_1+t_2}$. To make $b$ a polynomial, Equation \eqref{eq: case 4.3 & 4.4} implies that $\ti g_2|(t+s_1 \ti f_2^2)$, $\ti g_2| a$ and hence $b$ is a polynomial of degree 1. Without loss of generality, we can assume that $\ti g_2$ has root $t_1$. Thus, let $a=t-t_1$ and then $\ti g_2^2=s_2(t-t_1)^2$ for some nonzero $s_2\in k$. Therefore,
$$a=t-t_1,\quad b=\frac{t+s_1 \ti f_2^2}{s_1 \ti g_2^2}a=\frac{t-t_2}{s_1s_2(t_1+t_2)}, \quad  c=f_1^2a=s_1\ti g_2^2a=s_1 s_2 (t-t_1)^3.$$

\end{proof}

\bibliographystyle{alpha}
\bibliography{bibfile}
\end{document}